\documentclass[letterpaper,10pt,reqno]{amsart}

\usepackage{graphicx}
\usepackage{caption}
\usepackage{subcaption}

\usepackage{indentfirst} 
\usepackage{amssymb}
\usepackage{mathtools} 
\usepackage{mathabx} 
\usepackage{amsthm}
\usepackage{thmtools}
\usepackage{enumitem} 
\usepackage{fancyhdr}
\usepackage{lipsum}
\usepackage[pagewise]{lineno}

\usepackage[colorlinks=true]{hyperref} 
\usepackage[usenames,dvipsnames]{xcolor} 
\usepackage{ifthen}
\usepackage{tikz}
\usetikzlibrary{decorations.pathreplacing}
\usepackage{tikz-cd}
  \usepackage{dsfont}

\makeatletter

\def\MRbibitem{\@ifnextchar[\my@lbibitem\my@bibitem}

\def\mybiblabel#1#2{\@biblabel{{\hyperref{http://www.ams.org/mathscinet-getitem?mr=#1}{}{}{#2}}}}

\def\myhyperanchor#1{\Hy@raisedlink{\hyper@anchorstart{cite.#1}\hyper@anchorend}}

\def\my@lbibitem[#1]#2#3#4\par{%
  \item[\mybiblabel{#2}{#1}\myhyperanchor{#3}\hfill]#4%
  \@ifundefined{ifbackrefparscan}{}{\BR@backref{#3}}%
  \if@filesw{\let\protect\noexpand\immediate
    \write\@auxout{\string\bibcite{#3}{#1}}}\fi\ignorespaces%
}

\def\my@bibitem#1#2#3\par{%
  \refstepcounter\@listctr
  \item[\mybiblabel{#1}{\the\value\@listctr}\myhyperanchor{#2}\hfill]#3%
  \@ifundefined{ifbackrefparscan}{}{\BR@backref{#2}}%
  \if@filesw\immediate\write\@auxout
    {\string\bibcite{#2}{\the\value\@listctr}}\fi\ignorespaces%
}

\makeatother

\DeclareFontFamily{U} {MnSymbolA}{}
\DeclareFontShape{U}{MnSymbolA}{m}{n}{
   <-6> MnSymbolA5
   <6-7> MnSymbolA6
   <7-8> MnSymbolA7
   <8-9> MnSymbolA8
   <9-10> MnSymbolA9
   <10-12> MnSymbolA10
   <12-> MnSymbolA12}{}
\DeclareFontShape{U}{MnSymbolA}{b}{n}{
   <-6> MnSymbolA-Bold5
   <6-7> MnSymbolA-Bold6
   <7-8> MnSymbolA-Bold7
   <8-9> MnSymbolA-Bold8
   <9-10> MnSymbolA-Bold9
   <10-12> MnSymbolA-Bold10
   <12-> MnSymbolA-Bold12}{}
\DeclareSymbolFont{MnSyA} {U} {MnSymbolA}{m}{n}
 \DeclareFontFamily{U} {MnSymbolC}{}
\DeclareFontShape{U}{MnSymbolC}{m}{n}{
  <-6> MnSymbolC5
  <6-7> MnSymbolC6
  <7-8> MnSymbolC7
  <8-9> MnSymbolC8
  <9-10> MnSymbolC9
  <10-12> MnSymbolC10
  <12-> MnSymbolC12}{}
\DeclareFontShape{U}{MnSymbolC}{b}{n}{
  <-6> MnSymbolC-Bold5
  <6-7> MnSymbolC-Bold6
  <7-8> MnSymbolC-Bold7
  <8-9> MnSymbolC-Bold8
  <9-10> MnSymbolC-Bold9
  <10-12> MnSymbolC-Bold10
  <12-> MnSymbolC-Bold12}{}
\DeclareSymbolFont{MnSyC} {U} {MnSymbolC}{m}{n}

\DeclareMathSymbol{\top}{\mathord}{MnSyA}{219} 
\DeclareMathSymbol{\plus}{\mathord}{MnSyC}{20} 

\declaretheorem[numberwithin=section]{theorem}
\declaretheorem[sibling=theorem]{lemma}

\declaretheorem[sibling=theorem]{proposition}
\declaretheorem[sibling=theorem,style=definition]{definition}
\declaretheorem[sibling=theorem]{remark}
\declaretheorem[sibling=theorem]{example}

\hypersetup{bookmarksdepth = 3} 
\numberwithin{equation}{section}     

\setlist[enumerate,1]{label={\upshape(\alph*)},ref=\alph*}
\setlist[enumerate,2]{label={\upshape(\arabic*)},ref=\arabic*}

\newcommand{\R}{\mathbb{R}}

\newcommand{\N}{\mathbb{N}}

\def\phi{\varphi}
\def\R{{\mathbb R}}

\def\N{{\mathbb N}}

\newcommand{\revised}[1]{\textcolor{blue}{#1}}
\newcommand{\old}[1]{\textcolor{gray}{#1}}

\newcommand{\vertiii}[1]{{\left\vert\kern-0.25ex\left\vert\kern-0.25ex\left\vert #1 
    \right\vert\kern-0.25ex\right\vert\kern-0.25ex\right\vert}}
\newcommand{\invertiii}[1]{{\vert\kern-0.25ex\vert\kern-0.25ex\vert #1 
    \vert\kern-0.25ex\vert\kern-0.25ex\vert}}

\begin{document}

\title{Equilibrium states for piecewise weakly convex interval maps}

\date{\today}


\thanks{Supported by ANID Doctorado Nacional 21210037.}

\author{ Nicol\'as Ar\'evalo-Hurtado.}	
\address{Facultad de Matem\'aticas,
Universidad Escuela Colombiana de Ingenier\'ia Julio Garavito, Ak. 45 \# 205-59 (Autopista Norte), Bogot\'a, Colombia}
\email{\href{nicolas.arevalo-h@escuelaing.edu.co}{nicolas.arevalo-h@escuelaing.edu.co}}
\urladdr{\url{https://sites.google.com/view/nicolasarevalomath/}}

\begin{abstract}
We prove the existence of equilibrium states for geometric potentials in a class of piecewise weakly convex interval maps. This class includes systems with indifferent fixed points and non-Markov partitions. Under additional hypotheses we also obtain uniqueness.
\end{abstract}

\maketitle

\section{Introduction}
Let $T:[0,1]\rightarrow [0,1]$ be a \textit{piecewise monotone, orientation-preserving transformation}; that is, there is a fixed partition $0=a_{0}<\cdots < a_{N}=1$ such that for each $1\leq k \leq N$ the restriction $T_{k}:=T|_{I_{k}}$, where $I_{k}=[a_{k-1},a_{k})$, is strictly increasing and continuous. Suppose furthermore that $T$ is non-singular, i.e., $m\circ T^{-1}_{k}$ is absolutely continuous with respect to $m$ for each $1\leq k\leq N$, where $m$ is the Lebesgue measure. Since each branch $T_{k}$ is injective, an extended inverse function can be defined on $[0,1]$: for each $x\in [0,1]$
\begin{align*}
\psi_{k}(x):=\begin{cases}a_{k-1} & \text{if $x\leq \inf_{y\in I_{k}} T_{k}(y)$},\\
T^{-1}_{k}(x) & \text{if $x\in T_{k}\mathring{I}_{k}$},\\
a_{k} & \text{if $x\geq \sup_{y\in I{k}} T_{k}(y)$},\end{cases}
\end{align*}where $\mathring{I_{k}}=(a_{k-1},a_{k})$. Since $\psi_{k}$ is differentiable for Lebesgue almost every point, we may assume that for each $1\leq k \leq N$ the function $\sum^{k}_{i=1}\psi_{i}'$ is upper semicontinuous. In \cite{bmss}, Bose, Maume-Deschamps, Schmitt, and Shin considered non-singular piecewise monotone and orientation-preserving maps such that
\begin{enumerate}
    \item \label{Cond1} $T'(0)>1$ and $T_{k}'(a_{k-1})>0$ for each $1\leq k \leq N$.
    \item \label{Cond2} the function $\sum^{k}_{i=1}\psi_{i}'$ is non-increasing for each $1\leq k\leq N$.
\end{enumerate}
For each map $T$ satisfying these two conditions, the existence of a $T$-invariant measure absolutely continuous with respect to the Lebesgue measure (ACIM) was established (see \cite[Theorem 1.1]{bmss}). For such maps the Perron--Frobenius operator preserves the cone of non-increasing, non-negative functions on the unit interval. Recall that fixed points of the Perron--Frobenius operator correspond to ACIMs with respect to $m$. This generalizes Lasota and Yorke's results on a class $\mathcal{C}$ of piecewise monotone, orientation-preserving transformations in which each $T_{k}$ is convex, $T_{k}(a_{k-1})=0$, and $T'(0)>1$ (see Example \ref{ConvexTransformationsLasotaYorke} and \cite{ly}). In that setting, the Perron--Frobenius operator preserves the space of functions of bounded variation. Our objective is to establish the existence of equilibrium states for geometric potentials for a subfamily satisfying conditions (\ref{Cond1}) and (\ref{Cond2}) while still containing $\mathcal{C}$.
 
Since we may work with different ergodic measures, we need more than the a.e.\ existence of $\psi_{i}'$. We say that $T$ is an \textit{$a$-convex} transformation if it satisfies condition (\ref{Cond1}) and, for each $1\leq k\leq N$, the function $\sum^{k}_{i=1}\psi_{i}$ is differentiable except at most at countably many points, and $\sum^{k}_{i=1}\psi_{i}'$ is non-increasing.

The weak convexity of these maps stems from the fact that if $T_{i}$ is convex, then $\psi_{i}$ is concave, so $\psi_{i}'$ exists except possibly on a countable set and is non-increasing. This family, referred to as $a$-convex transformations due to their average convexity, may include concave branches (so some $\psi_{i}'$ may be increasing), indifferent fixed points, and non-Markov partitions. These features limit the applicability of conjugacies with topological Markov shifts or of hyperbolic techniques, where thermodynamic formalism is more fully developed. Nevertheless, $a$-convex transformations share ordering properties with the family of attractive $g$-functions defined on the full shift with finitely many symbols (see \cite{hu} and Remark \ref{PHulseRemark}).

Let \( \varphi: X \to \mathbb{R} \) be a measurable function on the Borel $\sigma$-algebra; the \emph{topological pressure of $\varphi$} is defined as
\[
P(\phi) := \sup \left\{ h_{\mu} + \int \phi \, d\mu : \mu \in \mathcal{M}(X,T), \int \phi \, d\mu > -\infty \right\},
\]
where \( h_{\mu} \) is the measure-theoretic entropy of \( \mu \), and \( \mathcal{M}([0,1]) \) denotes the set of \( T \)-invariant probability measures. A measure that attains the supremum is called an \textit{equilibrium state} for \( \varphi \). We also refer to $\varphi$ as the \textit{potential}.

Let $s\in \R$, and suppose $\mu_{s}$ is an equilibrium state for the geometric potential $-s\log|T'|$. It is well known that, since $\mu_{s}$ is ergodic, the Lyapunov exponent of $\mu_{s}$-almost every point is constant. Consequently, equilibrium states of geometric potentials are related to the possible values that the Lyapunov exponent can attain (see \cite{io,an,we}). In the hyperbolic context, the ACIM with respect to the Lebesgue measure is an equilibrium state for the potential $-\log|T'|$ (see \cite[Theorem 9.7.1]{BG}). We aim to develop Thermodynamic Formalism for geometric potentials of $a$-convex transformations, i.e., study the space of equilibrium states for the family of geometric potentials.

Let $C([0,1])$ be the space of real-valued continuous functions on $[0,1]$. For each $s\in \R$ and each $f\in C([0,1])$, we define the operator
\begin{align*}
    F_{s}(f)=\sum^{N}_{i=1}(\psi_{i}')^{s}f\circ \psi_{i}.
\end{align*} In different parts of the paper we consider $F_{s}$ on different function spaces. For example, in the uniformly hyperbolic setting, fixed points may be sought in the space of H\"older continuous functions. When $s=1$, we can extend $F_{1}$ to $L^{1}_{m}([0,1])$, the space of integrable functions, where it coincides with the Frobenius--Perron operator (see \cite{ly}).

 Following \cite{hk1}, in Section 3 we define an extension $\overline{[0,1]}$ of $[0,1]$ with the order topology such that the corresponding extension of $F_{s}$ preserves the space of continuous functions and $T'$ admits a continuous extension. We establish the existence of a \textit{conformal measure} $m_{s}$ related to $F_{s}$ on $\overline{[0,1]}$, i.e., for each Borel set $A$ on $\overline{[0,1]}$ and any $f\in L^{1}_{m_{s}}(\overline{[0,1]})$
 \begin{align*}
     \int_{T^{-1}A}fdm_{s}=\int_{A}F_{s}fdm_{s},
 \end{align*}up to a constant multiplication.
 

 There are many strategies for studying the thermodynamic formalism of non-uniformly hyperbolic systems, including inducing schemes and Hofbauer towers (see \cite{li}). Nevertheless, most approaches require Markov partitions. For systems with non-Markov partitions, the lack of a symbolic framework often restricts thermodynamic formalism to locally invertible systems (see \cite{hk1}). Thus, dynamical systems with both indifferent fixed points and non-Markov partitions are of particular interest.
 
For expanding, piecewise monotone, orientation-preserving maps on $[0,1]$, Keller and Hofbauer found fixed points of the operator $F_{s}$ on the space $L^{1}_{m_{s}}$ of $m_{s}$-integrable functions and showed that these fixed points yield the density of the equilibrium state for the geometric potential. Their arguments do not apply directly in the presence of indifferent fixed points in some $a$-convex transformations (see Example \ref{ExampleParabolicAconvex}). Moreover, even when indifferent fixed points exist, we do not employ inducing to study the thermodynamic formalism of $a$-convex transformations; a different approach is needed to handle non-Markov partitions. Finally, the geometric potential need not be H\"older continuous (see Example \ref{ExampleParabolicNonMarkovianAconvex}), a condition usually required for the existence of equilibrium states. We now state sufficient conditions for our results.

When studying $a$-convex transformations, expansivity at $t=0$ causes the existence of a forward invariant interval. Given $T$ an $a$-convex transformation, there exists a point $\beta\in (0,1]$ such that $\bigcup^{\infty}_{i=0}T^{i}[0,a_{1}]=[0,\beta]$ (see \cite[Lemma 2]{bmss}). Without loss of generality, we may assume that $\beta=a_{N^{*}}$ for some  $1\leq N^{*}\leq N$ (see Remark \ref{SomeIterationExpansive}). For every $1\leq k\leq  N$, let $I
_{k}=[a_{k-1},a_{k})$, and for every array of integers $(d_{i})^{r}_{i=1}\in \{1,...,N\}^{r}$, define the  \textit{cylinder}
\begin{align*}
    I_{(d_{i})^{r}_{i=1}}=\bigcap^{r-1}_{i=0}T^{-i}I_{d_{i+1}}.
\end{align*} Note that $I_{(d_{i})^{r}_{i=1}}$ may be empty. 

For every probability measure $\mu$ and every measurable function $f$ we denote $\mu(f)=\int f d\mu$. Now we consider $F_{s}$ to be defined over the set of bounded variation functions. We say that $T$ satisfies \textit{condition $(B)$} for $s>0$ if there exists a sequence of positive numbers $\{M_{r}\}_{r\in \N}$ that converges to zero such that for every cylinder $I_{(d_{i})^{r}_{i=1}}$ that does not contain $\beta$, we have
\begin{align*}
    \inf_{n\in \N} \dfrac{\|F^{n}_{s}\chi_{I_{(d_{i})^{r}_{i=1}}}\|_{\infty}}{(m_{s}(F_{s}\mathds{1}))^{n}}\leq M_{r},
\end{align*} where $\|\cdot\|_{\infty}$ denotes the supremum norm and $\mathds{1}$ the constant function equal to 1. Let $\mathcal{J}$ be the cone of non-negative non-increasing functions over $[0,1]$. We say $T$ satisfies \textit{condition $(C)$} for $s>0$ if for every $1\leq k \leq N$ we have $\sum^{k}_{i=1}(\psi_{i}')^{s}\in \mathcal{J}$.

\begin{theorem}\label{PrincipalThmChap52}
    Let $T$ be an $a$-convex transformation satisfying conditions $(B)$ and $(C)$ for a given $s>0$. If either 
    \begin{align*}
        \lim_{x\rightarrow \beta^{-}}\psi_{N^{*}}'(x)<1\text{, or  $\lim_{x\rightarrow \beta^{+}}\psi_{N^{*}+1}(x)=1$},
    \end{align*} then there exists an equilibrium state $\mu_{s}$ for the potential $-s\log|T'|$ which is absolutely continuous with respect to $m_{s}$, and such that $\frac{d\mu_{s}}{dm_{s}}\in \mathcal{J}$. Furthermore, if $\beta=1$ and $m_{s}\neq \delta_{\beta}$, then $\mu_{s}$ is the unique equilibrium state for $-s\log|T'|$, which is absolutely continuous with respect to $m_{s}$. 
\end{theorem}
This article is organized as follows. Section 2 presents examples of $a$-convex transformations not covered in \cite{bmss}, illustrating different dynamics to which Theorem \ref{PrincipalThmChap52} applies. Section 3 addresses the existence of the conformal measure $m_{s}$ for the operator $F_{s}$. Section 4 studies the space of $F_{s}$-invariant functions in the space of functions of bounded variation. Finally, Section 5 contains the proof of Theorem \ref{PrincipalThmChap52}.

\section{On $a$-convex transformations}

Given an $a$-convex transformation $T$, we say that a fixed point $x\in (0,1]$ is \textit{indifferent} if $|T'(x)|=1$. The partition defined by $\{a_{i}\}_{i=0}^{N}$ is called a \textit{Markov partition} if for any $1\leq i,j \leq n$ such that $(a_{i-1},a_{i})\cap T((a_{j-1},a_{j}))$ is non-empty, then $(a_{i-1},a_{i})\subset T((a_{j-1},a_{j})).$ 
\begin{remark}\label{SomeIterationExpansive} \normalfont
    Here we summarize the necessary properties of the point $\beta$ mentioned in the introduction to prove our results.
    
    Let $T$ be an $a$-convex transformation. Then, there exists a point $\beta \in (0,1]$ and a positive integer $r$ such that  (see \cite[Lemma 4.1 and Lemma 4.2]{bmss})
    \begin{align}\label{ZeroBetaForwardInvariant}
        \bigcup^{r}_{k=0}T^{k}[0,a_{1}]=\bigcup^{\infty}_{k=0}T^{k}[0,a_{1}]=[0,\beta].
    \end{align}
    Moreover, let $N^{*} \leq N$ be such that $\beta \in [a_{N^{*}-1}, a_{N^{*}}]$. We can split the $N^{*}$-th branch of $T$ so that $T|_{[0,\beta]}$ is again an $a$-convex transformation on the interval $[0,\beta]$. From now on, we assume $\beta=a_{N^{*}}$.
    
    The subsystem $(T, [0, \beta])$ carries all the relevant dynamical properties, as it inherits the hyperbolicity of $T$ at $x=0$. If $\lim_{x\rightarrow \beta^{-}}\psi_{N^{*}}'(x)<1$, then $(T,[0,\beta])$ is exact, has exponential decay of correlations, and some iterate of $T$ is expansive on $[0,\beta]$ (see \cite[Theorem 4.5, Remark 4.2]{bmss}).

\end{remark}
Recall that, for each measurable set $A$ and for $f\in L^{1}_{m}([0,1])$, the Frobenius-Perron operator $F_{1}$ with respect to the Lebesgue measure $m$ satisfies
\begin{align}\label{Perron-Frobenius}    \int_{T^{-1}A}fdm=\int_{A}F_{1}fdm.
\end{align} 
\begin{example}\label{ExampleParabolicAconvex}
    For every $x\in [0,1]$, consider the map defined by (see Figure \ref{fig:$a$-convexParabolic})
    \begin{align*}
        T(x)=\begin{cases}
            \frac{1}{3}(2+3x-2\sqrt{1-3x}) & \text{if $x\in [0,1/3)$,}\\
            x-\frac{3}{4}(1-x)^{2} & \text{if $x\in [1/3,1]$}.
        \end{cases}
    \end{align*}
    
    This example is constructed to satisfy 
\begin{align}\label{ExamploParaAconvex}
    \psi_{1}'=1-\psi_{2}'.
\end{align}
By the Inverse Function Theorem, this implies
\begin{align*}
    -\dfrac{T_{1}''\circ T_{1}^{-1}}{(T_{1}'\circ T_{1}^{-1})^{3}}=\dfrac{T_{2}''\circ T_{2}^{-1}}{(T_{2}'\circ T_{2}^{-1})^{3}}.
\end{align*}
Thus, $T_{1}$ is convex, $T_{2}$ is concave, and equation \eqref{ExamploParaAconvex} implies that $\psi_{1}'+\psi_{2}'=1$. Note also that $T_{1}'(0)>1$. Consequently, $T$ is an $a$-convex transformation, and the Lebesgue measure is $T$-invariant because $F_{1}\mathds{1}=\mathds{1}$ (see equation \ref{Perron-Frobenius}). Moreover there is an indifferent fixed point at $x=1$, which implies $\lim_{x\rightarrow 1/3^{-}}T_{1}'(x)=\infty$ by equation \eqref{ExamploParaAconvex}. 

    \begin{figure}\label{ExampleFigure 1}
     \centering
     \begin{subfigure}[b]{0.45\textwidth}
         \centering
         \includegraphics[width=\textwidth]{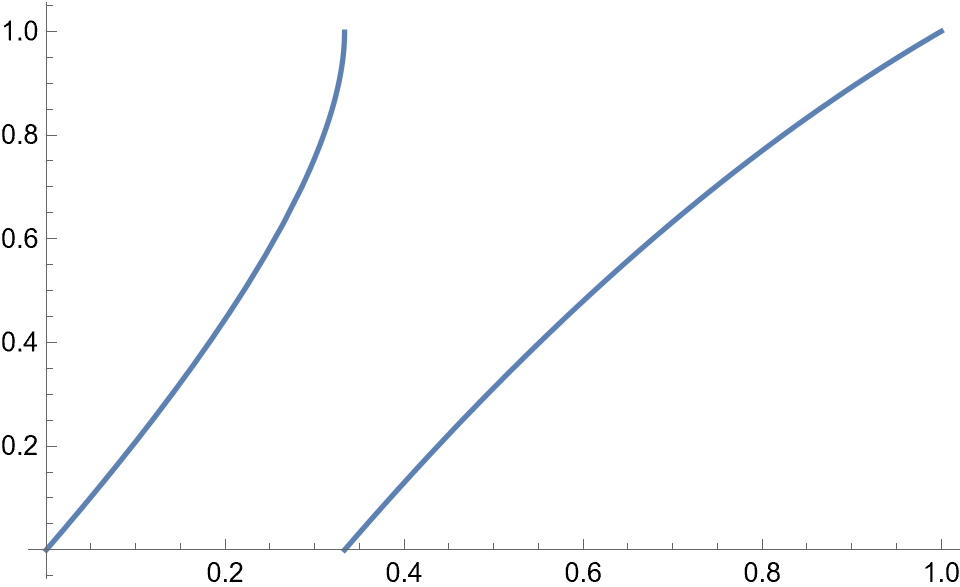}
         \caption{Parabolic fixed point at $x=1$.}
     \end{subfigure}
     \hfill
     \begin{subfigure}[b]{0.45\textwidth}
         \centering
         \includegraphics[width=\textwidth]{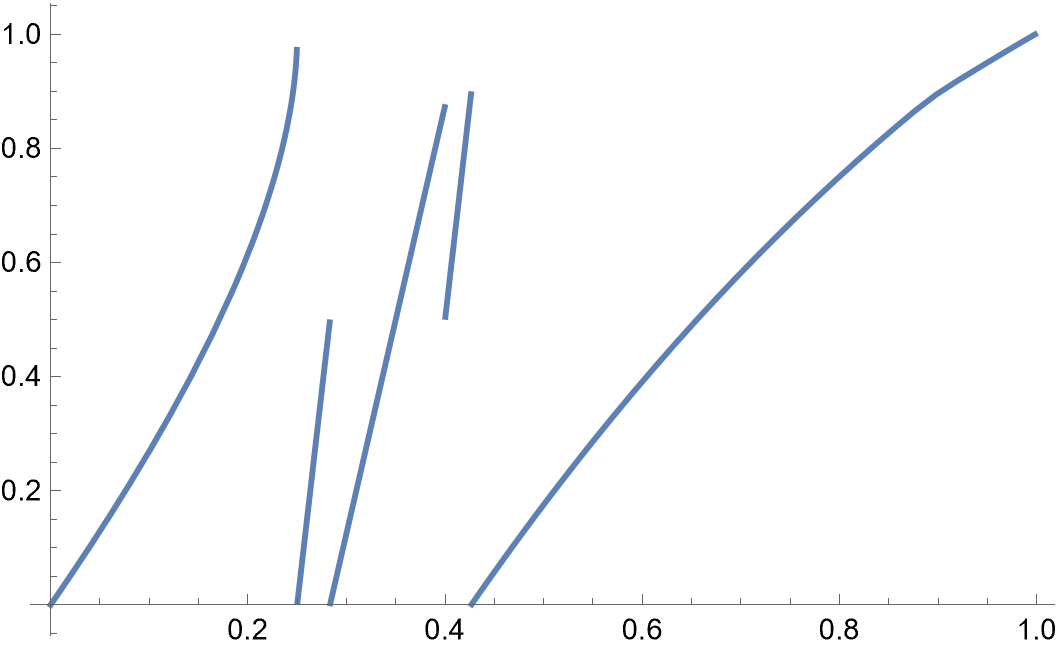}
         \old{\caption{Parabolic fixed point and non-markov partition.}}
         \revised{\caption{Parabolic fixed point and non-Markov partition.}}
     \end{subfigure}
        \caption{The graph of $a$-convex transformations from Example \ref{ExampleParabolicAconvex} (A) and Example \ref{ExampleParabolicNonMarkovianAconvex} (B).}
        \label{fig:$a$-convexParabolic}
\end{figure}
  Since $\psi_{1}'$ is non-increasing, it follows that $(\psi_{1}')^{s}$ is also non-increasing for every $s>0$. Additionally, for $s>0$ with $s\neq 1$, we have 
\begin{align*}
    \dfrac{d}{dx}((\psi_{1}')^{s}+(\psi_{2}')^{s})=s\psi_{2}''(-(1-\psi_{2}')^{s-1}+(\psi_{2}')^{s-1}).
\end{align*}
Because $\psi_{2}''>0$, the function $(\psi_{1}')^{s}+(\psi_{2}')^{s}$ is non-increasing if and only if 
\begin{align}\label{Ecu537}
    (\psi_{2}')^{s-1}-(1-\psi_{2}')^{s-1}\leq 0.
\end{align}
Since $\psi_{2}'$ is increasing and $\psi_{2}'(0)=1/2$, we have $\psi_{2}'\geq 1/2$. Therefore, equation \eqref{Ecu537} holds only for $0<s<1$. Consequently, $T$ is an $a$-convex transformation that satisfies condition $(C)$ for all $0<s\leq 1$. In Example \ref{ConditionBExamplesInfFixPt}, we show that this map also satisfies condition $(B)$ for $0<s<1$.\end{example}
\begin{example}\label{ExampleParabolicNonMarkovianAconvex}
      We can also construct an $a$-convex transformation with an indifferent fixed point at $x=1$, and without a Markov partition (see Figure \ref{fig:$a$-convexParabolic}). For each $x\in [0,1]$ consider 
      \begin{align*}
         T(x)=\begin{cases}
             \frac{1}{4}\left(3-3\sqrt{1-4x}+4x\right) & \text{if $x\in [0,\frac{1}{4})$},\\
             15x-\frac{15}{4} & \text{if $x\in [\frac{1}{4},\frac{17}{60})$},\\
             \frac{15}{2}x-\frac{17}{8} & \text{if $x\in [\frac{17}{60},\frac{2}{5})$},\\
             15x-\frac{11}{2} & \text{if $x\in [\frac{2}{5},\frac{32}{75})$}.
         \end{cases}
     \end{align*}
      The branches $T_{2},T_{3}$, and $T_{4}$ are constructed so that $\sum_{i=1}^{k}\psi_{i}'<1$ and $\sum_{i=1}^{k}\psi_{i}'\in \mathcal{J}$ for each $1\leq k\leq 4$. Note also that $T'(0)>1$. Finally, $T_{5}$ is implicitly defined on $[\frac{32}{75},1]$ by the equation $\sum_{i=1}^{5}\psi_{i}'=1$. Thus, $T$ is an $a$-convex transformation, and the Lebesgue measure is $T$-invariant. Since $T[\frac{17}{60},\frac{2}{5}]=[0,\frac{7}{8}]$ is not a union of intervals in the partition $\mathcal{P}=\{[a_{j-1},a_{j})\}_{j=1}^{5}$, $\mathcal{P}$ is non-Markov. Furthermore, $x=1$ is an indifferent fixed point because $\sum_{j=1}^{5}\psi'_{j}(1)=0+\psi_{5}'(1)=1$.
\end{example}

\section{Conformal measure for $-s\log|T'|$}
Let $T$ be an $a$-convex transformation satisfying condition $(C)$ for some $s > 0$ and $h \in C([0,1])$. The \textit{transfer operator} of $T$ with respect to $h$ is defined on $C([0,1])$ as follows: for each $f\in C([0,1])$
\begin{align*}
    f(x)\mapsto \sum_{y\in T^{-1}x}e^{h(y)}f(y).
\end{align*} Notice that $\psi_{i}'(x)= e^{-\log|T'T_{i}^{-1}x|}\chi_{\overline{TI_{i}}}$ for each $1\leq i \leq N$, where $\chi_{\overline{TI_{i}}}$ denotes the characteristic function of the closure of $T(I_{i})$.

In this section we construct an extension $\overline{[0,1]}$ of $[0,1]$ as a topological space in which $T'$ and the function $\sum_{i=1}^{k} (\psi_{i}')^{s}$ is continuous for every $1 \leq k \leq N$. Thus, for every $f\in C(\overline{[0,1]})$ the extended operator $F_{s}$ is as the transfer operator for the potential $h=-s\log|T'|$ on $\overline{[0,1]}$. We follow the construction presented in \cite{hk1} to achieve this. Define
\begin{align*}
            V_{k}:=\left\{x\in[0,1]:\lim_{y\rightarrow x^{+}} \sum^{k}_{i=1}\psi_{i}'(y)\neq \lim_{y\rightarrow x^{-}}\sum^{k}_{i=1}\psi_{i}'(y)\right\}.
\end{align*} Let $V=\left(\bigcup^{N}_{k=1}V_{k}\right)\cup \{0,1\}$. Then, all possible orbits with discontinuities of $\sum^{k}_{i=1}(\psi_{i}')^{s}$ and $T$ are contained in the set 
        \begin{align*}
            W=\bigcup^{\infty}_{i=0}T^{-i}\left(\bigcup^{\infty}_{j=0}T^{j}V\right)\setminus \{0,1\}.
        \end{align*}
Since $V$ is countable by the differentiability condition of $a$-convex transformations, it follows that $W$ is a countable $T$-invariant set. For each $x\in W$, we replace $x$ in $[0,1]$ with $x^{+}$ and $x^{-}$ to define
        \begin{align*}
            \overline{[0,1]}=W^{c}\cup \bigcup_{x\in W}\{x^{+}\}\cup \bigcup_{x\in W}\{x^{-}\}.
        \end{align*}
The order relation on $[0,1]$ is extended to $\overline{[0,1]}$ as follows: $ y<x^{-}<x^{+}<z$ whenever $y<x<z$ in $[0,1]$. The space $\overline{[0,1]}$ endowed with the order topology is a compact topological space (\cite[Theorem 12, Chapter X, section 7]{bg}). Since $[0,1]\setminus W$ is dense in $\overline{[0,1]}$, for every $x\in W$ we define $\sum^{k}_{i=1}(\psi_{i}')(x^{+})$, $\sum^{k}_{i=1}(\psi_{i}')(x^{-})$, $T(x^{+})$, and $T(x^{-})$ continuously on $\overline{[0,1]}$. Indeed, for any $x,y\in W$, we have 
    \begin{align*}
        [y^{+},x^{-}]=\{z\in \overline{[0,1]}:y^{+}\leq z\leq x^{-}\}=\{z\in \overline{[0,1]}:y^{-}<z<x^{+}\}=(y^{-},x^{+}),
    \end{align*} 
    which implies that the extensions of $\sum_{i=1}^{k}(\psi_{i}')^{s}$ and $T$ are continuous with respect to the order topology. Thus, the extended $F_{s}$ acts as an operator from $C(\overline{[0,1]})$ to $C(\overline{[0,1]})$. Since $F_{s}$ is linear, for any $f, g \in C(\overline{[0,1]})$, we have
    \begin{align*}
        \|F_{s}f-F_{s}g\|_{\infty}\leq K\|f-g\|_{\infty},
    \end{align*} 
    where $K = \left\|\sum_{i=1}^{N}(\psi_{i}')^{s}\right\|_{\infty}$ is finite by condition $(C)$. This implies that $F_{s}$ is continuous on $(C(\overline{[0,1]}),\|\cdot\|_{\infty})$. Therefore, the dual operator $F_{s}^{*}$ is continuous on $C^{*}(\overline{[0,1]})$, the dual space of $C(\overline{[0,1]})$. By the Riesz Representation Theorem, $C^{*}(\overline{[0,1]})$ is isomorphic to the Banach space of regular finite measures on $\overline{[0,1]}$ with the weak* topology. Since the space $\mathcal{M}_{1}(\overline{[0,1]})$ of positive Borel probability measures on $\overline{[0,1]}$ is convex and compact in $C^{*}(\overline{[0,1]})$, we apply the Schauder-Tychonoff Theorem to the continuous map
    \begin{align*}
        \mu \mapsto \dfrac{F_{s}^{*}\mu}{F_{s}^{*}\mu(\mathds{1})},\quad \mu\in \mathcal{M}_{1}(\overline{[0,1]}).
    \end{align*} 
    Thus, we obtain a fixed point $m_{s}$. In other words, for every $f \in C(\overline{[0,1]})$,
    \begin{align}\label{Msmeasure}
        m_{s}(F_{s}f)=\gamma m_{s}(f),
    \end{align}
    where $\gamma = m_{s}(F_{s}\mathds{1})$. Hence, $m_{s}$ is a fixed point of $\gamma^{-1}F^{*}_{s}$ and  $F_{s}\mathds{1}(1)\leq \gamma \leq F_{s}\mathds{1}(0)$. For simplicity $\overline{F}_{s}=\gamma^{-1}F_{s}$. The following proposition is analog to \cite[Lemma 6]{hk2}\label{PmarkovOp} and important to apply fixed points methods.

    \begin{proposition}\label{MarkovOp}
    Suppose that $m_{s}(\{x\}) = 0$ for every $x \in \overline{[0,1]}$. Then, $\overline{F_{s}}$ is a Markov operator on $L^{1}_{m_{s}}(\overline{[0,1]})$, i.e., for all $f \in L^{1}_{m_{s}}(\overline{[0,1]})$
        \begin{enumerate}
            \item $m_{s}(\overline{F_{s}}f)=m_{s}(f)$.
            \item $\|\overline{F_{s}}f\|_{L^{1}_{m_{s}}}\leq \|f\|_{L^{1}_{m_{s}}}$.
        \end{enumerate}
    \end{proposition}
     \begin{proof}
        Since $C(\overline{[0,1]})$ is dense in the Banach space $L^{1}_{m_{s}}(\overline{[0,1]})$, we may extend equation \eqref{Msmeasure} over $m_{s}$-integrable functions to get $(a)$. Let $f$ be a bounded measurable function and define in the support of $\overline{F}_{s}f$ the function $\varphi(x)=\frac{\overline{F}_{s}f(x)}{|\overline{F}_{s}f(x)|}$. Thus $|\overline{F}_{s}f(x)|=\overline{F}_{s}f(x)\cdot \varphi(x)$ and  $|\varphi(1)|=1$ a.e. $x\in m_{s}$, so that
\begin{align*}
    \|\overline{F}_{s}f\|_{L^{1}_{m_{s}}}=\int \overline{F}_{s}f(x)\cdot \varphi(x) dm_{s}= \int \overline{F}_{s}(f(\varphi\circ T) )dm_{s},
\end{align*} where we have used the property $\overline{F}_{s}(f \cdot (g\circ T))=\overline{F}_{s}f\cdot \varphi $ for any measurable bounded function $f$ and $g$. Since $m_{s}$ is a fixed point of $\overline{F}_{s}^{*}$ and $|\varphi (x)|=1$
\begin{align}\label{PsContracting1}
    \|\overline{F}_{s}f\|_{L^{1}_{m_{s}}}=\int f(\varphi \circ T)dm_{s} \leq \int |f|\cdot |\varphi \circ T| dm_{s} \leq \|f\|_{L^{1}_{m_{s}}}.
\end{align} 
If we change a bounded measurable function on a set of zero measure $B$, then $TB$ has zero measure. Thus $\overline{F}_{s}f$ will change in a set of zero measure. Then $\overline{F}_{s}$ can be extended continuously to all $L^{1}_{m_{s}}$.
     \end{proof}
    
    To deduce $m_{s}$ is indeed a probability measure on $[0,1]$, we require it to assign zero measure on the set of duplicate points $\bigcup_{x\in W}\{x^{+},x^{-}\}$. This is the case when $T^{n}$ is expanding for some $n\in \mathbb{N}$ (see \cite[Lemma 2]{hk1}). Bose et al. provide sufficient conditions on $a$-convex transformations for $T$ to have some expansive iteration (see \cite[Theorem 4.10]{bmss} and Remark \ref{SomeIterationExpansive}). To address cases where $(T,[0,1])$ has indifferent fixed points, condition $(B)$ will be sufficient to imply $m_{s}\in \mathcal{M}_{1}([0,1])$ (see Remark \ref{SomeIterationExp2}).

    \begin{example}\label{ConvexTransformationsLasotaYorke}

Consider the family $\mathcal{C}$ of piecewise monotone convex transformations studied in \cite{ly}. Let $T \in \mathcal{C}$. Since each $T_k$ is convex, it is differentiable except at most in countably many points. Therefore, for each $1 \leq i \leq N$, we have $\psi_i' \in \mathcal{J}$ up to a countable set. This implies that $\sum_{i=1}^k (\psi_i')^s \in \mathcal{J}$ for every $1 \leq k \leq N$ and any $s>0$, which shows that $T$ is an $a$-convex transformation satisfying condition $(C)$ for all $s > 0$.
    
We now show that $T^n$ is expansive for some $n \in \mathbb{N}$. Notice that either $T(\beta) = \beta$ or $\psi_{N^{*}}'(\beta^{-}) = 0$. The second case satisfies our requirement. Suppose $T(\beta) = \beta$. By convexity, we have
    \begin{align*}
        1< \dfrac{\beta}{\beta-a_{N^{*}-1}}\leq T'(\beta^{-})=(\psi_{N^{*}}'(\beta^{-}))^{-1},
    \end{align*}since $T_{N^{*}}(a_{N^{*}-1})=0$. Hence, $\psi_{N^{*}}(\beta^{-})<1$.  By Remark \ref{SomeIterationExpansive}, 
 there exists $n_0 \in \mathbb{N}$ and $\alpha>1$ such that $(T^{n_0})' > \alpha $.  Therefore, $T$ satisfies condition $(B)$ by \cite[Lemma 2]{hk1} and equation \eqref{Msmeasure} applied to $\chi_{I_{(d_{i})_{i=1}^{r}}}$ the characteristic function of any cylinder $I_{(d_{i})_{i=1}^{r}}$.
\end{example}
\begin{remark}\label{SomeIterationExp2} \normalfont
    Now, we discuss why $\psi_{N^{*}}'(\beta) \geq 1$ implies that the point $\beta^{-}$ is indifferent. First, since $\sum^{N^{*}}_{i}\psi_{i}'\in \mathcal{J}$ it implies that $\sum_{i=1}^{N^{*}} \psi_{i}'(x) \geq 1$ for any $x \in [0,\beta]$. Integrating both sides with respect to Lebesgue measure, we obtain for any $x \in [0,\beta]$:
\begin{align}\label{QQ}
    \sum^{N^{*}}_{i=1}(\psi_{i}(x)-a_{i-1})\geq x.
\end{align} Hence, we have $\psi_{N^{*}}(\beta) \geq \beta$ when evaluated at $x = \beta$. Since $[0, \beta]$ is forward invariant by equation \eqref{ZeroBetaForwardInvariant} we must have equality on inequality \eqref{QQ}. So $\lim_{x\rightarrow \beta^{-}}T_{N^{*}}(x)=\beta$, and $\sum_{i=1}^{N^{*}} \psi_{i}'(\beta) = 1$. Be warned that we can have that $T$ is not continuous at $\beta$, so $T(\beta^{-})$ may differ from $T(\beta^{+})$. Therefore $\psi_{N^{*}}'(\beta^{-}) = 1$, meaning that $\beta^{-}$ is an indifferent fixed point. Furthermore, by equation \eqref{Perron-Frobenius} $F_{1} \chi_{[0,\beta]} = \chi_{[0,\beta]}$, so $\beta^{-1} \chi_{[0,\beta]}$ is the density of an ACIM for the Lebesgue measure. 

This argument also implies that $\psi_{N^{*}}'(x) <1$ for all $x \in [a_{N^{*}-1}, \beta)$. Indeed, if there exists an $a_{N^{*}-1} \leq x < \beta$ such that $\psi_{N^{*}}(x) = 1$, then we must have $\psi_{i}'(x) = 0$ for every $i < N^{*}$. In this case, $T_{i}[a_{i-1},a_{i}] \subset [0,x)$, which contradicts the forward invariance of $[0,\beta]$.
\end{remark}
 Define $\mathcal{P} = \{[a_{k-1}^{+}, a_{k}^{-}]\}_{k=1}^{N}$. For each $n \in \mathbb{N}$, the partition $\mathcal{P}^{n} = \bigvee_{i=0}^{n-1} T^{-i} \mathcal{P}$ is the finite partition of $\overline{[0,1]}$ into intervals on which $T^{n}$ is monotone and continuous, where $\bigvee$ denotes the join operator for partitions. Recall that $\mathcal{P}$ is said to be a generating partition for $m_{s}$ if $\bigvee_{i=0}^{\infty} T^{-i} \mathcal{P}$ is the partition into points of $\overline{[0,1]}$ up to zero $m_{s}$-measure. Typically, the argument that $\mathcal{P}$ is a generating partition is established by the existence of an expansive iteration $T^{k}$ (Lemma \cite[Lemma 2]{hk1}), a property that we do not necessarily have when $\beta^{-}$ is indifferent.

\begin{proposition}\label{IKPartition}
    Suppose $T$ satisfies condition $(B)$ and $(C)$ for some $s > 0$. Then, either $m_{s}\in\{\delta_{\beta^{-}},\delta_{\beta^{+}}\}$, or $\mathcal{P}$ is a generating partition for $m_{s}$.
    \end{proposition}
\begin{proof}
Let $r \in \mathbb{N}$ and $(d_{i})_{i=1}^{r} \in \{1, \dots, N\}^{r}$. Equation \eqref{Msmeasure} implies that for every $n \in \mathbb{N}$,
    \begin{align}\label{Ecu538}
        m_{s}(I_{(d_{i})_{i=1}^{r}})& =m_{s}(\overline{F_{s}}^{n}(\chi_{I_{(d_{i})_{i=1}^{r}}})).
    \end{align}Since $T$ satisfies condition $(B)$, for any $I_{(d_{i})_{i=1}^{r}}$ not containing $\beta$, we have \begin{align*}
        m_{s}(I_{(d_{i})_{i=1}^{r}})\leq M_{r}.
    \end{align*}Notice that it is enough to consider the case where $\lim_{x \to \beta^{-}} \psi_{N^{*}}(x) = 1$. If $T$ has an expansive iteration we can choose $\{M_{r}\}_{r \in \mathbb{N}}$ to decrease exponentially. On the other hand, assume that $\beta^{-}$ is an indifferent fixed point.
    
Define $w_{0} = 0$; for every $n \in \mathbb{N}$, define recursively $w_{n+1}$ by the equation $w_{n} = T_{N^{*}}(w_{n+1})$. Since $T' > 1$ over $[a_{N^{*}-1}, \beta)$, the sequence $\{w_{n}\}_{n \in \mathbb{N}}$ is strictly increasing and converges to $\beta$. Let $x \in \overline{[0,1]} \setminus \{\beta^{+}, \beta^{-}\}$. Then, either there exists $n \in \mathbb{N}$ such that $x \leq w_{n}$, or $x > \beta$. Denote by $I_{(d_{i})_{i=1}^{r}}(x)$ the cylinder of length $r$ containing $x$. If $r \geq n$, then by equation \eqref{Ecu538},
\begin{align*}
        m_{s}(I_{(d_{i})_{i=1}^{r}}(x)) \leq M_{r}.
    \end{align*}
The result follows whenever $m_{s}\notin \{ \delta_{\beta^{+}},\delta_{\beta^{-}}\}$.  
    \end{proof}
Thus, unless $m_{s}\in \{\delta_{\beta^{+}},\delta_{\beta^{-}}\}$, condition $(B)$ implies $m_{s}(\{x\}) = 0$ for all $x \in \overline{[0,1]}$. We can go further and say that, in case $m_{s}=\delta_{\beta^{+}}$, we have $\psi_{N^{*}+1}'(\beta^{+})=1$ and $\psi_{N^{*}}'(\beta^{-})=1$. This implies that $T'$ is continuous at $\beta$, so $\beta$ is an indifferent fixed point by Remark \ref{SomeIterationExpansive}.  Since the set $W$ is countable, we conclude that $m_{s} \in \mathcal{M}_{1}([0,1])$ whenever $\psi_{N^{*}}(\beta^{-})<1$ or $\psi_{N^{*}+1}'(\beta^{+})=1$. This explains assumptions of theorem \ref{PrincipalThmChap52}. 

\begin{example}\label{ConditionBExamplesInfFixPt}

    Consider $T$, the $a$-convex transformation from Example \ref{ExampleParabolicAconvex}. Let $w_1 = 1/3$ and $w_0 = 0$. For each $r \in \mathbb{N}$, define $w_{r+1}$ recursively by the equation $w_r = T_2(w_{r+1})$. Since $T_2$ is increasing, for any $x \in [w_r, w_{r+1})$, we have that $T^i x \in [w_{r-i}, w_{r+1-i})$ for all $1 \leq i \leq r$. Thus, the sequence $\{w_r\}_{r \in \mathbb{N}}$ is monotone increasing converging to 1.
    
    Moreover, there exist constants $C_3, C_4 > 0$ such that for every $x \in [w_r, w_{r+1})$ (see \cite[Proposicion 5.1.1]{mf}), we have
    \begin{align}\label{InequalityExampleParabolic}
        \dfrac{r^{2}}{C_{4}}\leq (T^{r})'(x) \leq \dfrac{r^{2}}{C_{3}}.
    \end{align}
    Recall that $T_1'$ is strictly increasing, and $T_2'(1/3^+) = 1/2 = T_1'(0)$. By inequality \eqref{InequalityExampleParabolic}, for any $x \leq w_{r+1}$, we obtain
    \begin{align*}
        \dfrac{r^{2}}{C_{4}} \leq (T^{r})'(w_{r+1}^{+}) \leq  (T^{r})'(x).
    \end{align*} 
    Let $r\in\mathbb{N}$ and let $I_{(a_i)_{i=1}^r}$ be any cylinder that does not contain $1$, i.e., $I_{(a_i)_{i=1}^r}\neq [w_r^+,1]$. By equation \eqref{Msmeasure},
    \begin{align*}
         m_{s}(I_{(a_{i})^{r}_{i=1}})=\int \overline{F}^{r}_{s}\chi_{I_{(a_{i})^{r}_{i=1}})}dm_{s}\leq \sup \||(T^{r})'|^{s}\gamma^{r}\|^{-1}_{\infty}\leq \dfrac{C_{4}}{r^{2s}\gamma{r}}.
    \end{align*} Recall that, condition $(C)$ implies $1=\sum^{N}_{i=1}(\psi_{i}'(1))^{s}\leq \gamma$. Defining $M_r = \frac{C_4}{r^2}$, we see that $T$ satisfies conditions $(B)$ and $(C)$ for $0 < s \leq 1$ whenever $m_s \neq \delta_1$. Analogous results can be shown for Example \ref{ExampleParabolicNonMarkovianAconvex}.
\end{example}

\section{Spectral Decomposition of the operator $\overline{F}_{s}$.}
In this section, we study the space of $\overline{F_s}$-invariant functions. For the remainder of the paper, we assume that $T$ satisfies the hypotheses of Theorem \ref{PrincipalThmChap52}.
\begin{definition}\label{BoundedVDefinition}
    We say that a function $f:[a,b]\rightarrow \R$ has \textit{bounded variation} on $[a,b]$ if there exists $M>0$ such that for any finite set $\{x_0, x_1, \dots, x_n\} \subset [a,b]$, the following inequality holds
        \begin{align*}
            \sum^{n}_{k=0}|f(x_{k})-f(x_{k+1})|\leq M.
        \end{align*}
    The \textit{variation} of $f$ on $[a,b]$ is defined by
        \begin{align*}
            V_{[a,b]}f=\sup\left\{\sum^{n-1}_{k=0}|f(x_{k})-f(x_{k+1})|: \text{$\{x_{0},x_{1},...,x_{n}\}\subset [a,b]$}\right\}.
        \end{align*}
\end{definition} 
\begin{proposition}\cite[Theorem 2.3.1.]{BG}\label{boundedIneq}
    If $f$ has bounded variation on $[a,b]$, then for all $x\in [a,b]$
        \begin{align*}
            |f(x)|\leq |f(a)|+V_{[a,b]}f.
        \end{align*}
    Furthermore, for any probability measure $\mu$ and for all $x \in [0,1]$, we have
    \begin{align*}
        |f(x)|\leq V_{[0,1]}f+\|f\|_{L^{1}_{\mu}}.
    \end{align*}
\end{proposition}
Therefore, the set of bounded variation functions, $BV[0,1]$, is a subspace of $L^1_\mu([0,1])$ for any probability measure $\mu$. This motivates us to look for a fixed point of $\overline{F_{s}}$ in the subspace $BV[0,1]$ (see also \cite{ly,BG}).

Recall that a real-valued function $g$ is said to be a density for $m_s$ if some $f$ in the class of $g$ satisfies $f\geq 0$ on a set of full measure and $m_{s}(f)=1$. To prove the existence of an ACIM for $m_s$, it is sufficient to take a suitable bounded variation function $f$ such that the iterations $\{\overline{F}_s^n f\}_{n \in \mathbb{N}}$ are contained in $BV[0,1]$, and the sequence $\{V_{[0,1]} \overline{F}_s^n f\}_{n \in \mathbb{N}}$ has an upper bound. Then, by applying Helly's First Theorem (see \cite[Theorem 2.3.9]{BG}), we show that $\{\overline{F}_s^n f\}_{n \in \mathbb{N}}$ is precompact in $L^1_{m_s}([0,1])$. The limit point $g$ of any convergent subsequence will be a fixed point of $\overline{F}_s$. Therefore, equation \eqref{Msmeasure} implies the existence of a $T$-invariant measure, absolutely continuous for $m_s$.

\begin{proposition}\label{LasotaYorkeInequality}
For every $f \in \mathcal{J}$, we have that $\overline{F}_s f \in \mathcal{J}$. Furthermore, for each $\alpha$ satisfying $\frac{1}{T'(0)} < \alpha < 1$, there exists a positive constant $b(\alpha)$ such that for all $f \in \mathcal{J}$,
    \begin{align}
            V_{[0,1]}\overline{F}_{s}f \leq \alpha V_{[0,1]}f+b \|f\|_{L^{1}_{m_{s}}}.
    \end{align}
\end{proposition} 
The proof of this proposition is analogous to \cite[Lemma 2.2]{bmss} and \cite[Lemma 2.4]{bmss} where the crucial property is $\overline{F}_{s}\mathds{1}\in \mathcal{J}$ instead of $\overline{F}_{1}\mathds{1}\in \mathcal{J}$. Inequality \eqref{LasotaYorkeInequality} is commonly referred to as a \textit{Lasota-Yorke type inequality}.

\begin{remark}\label{PHulseRemark}
We now explain the similarities between $a$-convex transformations and attractive $g$-functions considered in \cite{hu}. Let $(\Sigma,\sigma)=(\{1,...,N\}^{\N},\sigma)$ denote the fullshift on $N$ symbols. We write $x=(x_{i})_{i\geq 0}$ for elements in $\Sigma$. For each $1\leq j \leq N$, denote $[j]$ the set of points $x\in \Sigma$ such that $x_{0}=j$, and, for any $x\in \Sigma$ denote $\sigma^{-1}(x)\cap [j]$ by $jx$. We define a partial order on $\Sigma$ given by $x\leq y$ whenever $x_{i}\leq y_{i}$ for all non-negative integers $i$.

A function $g:\Sigma \rightarrow \R$ is said to be \textit{attractive} if there exists $\delta>0$ such that for all $x\in \Sigma$, we have $g(x)\geq \delta$, $\sum_{i=1}^{N}g(ix)=1$, and, for each $1\leq k \leq N$, the function $\sum^{N}_{i=k}g(ix)$ is non-decreasing. We obtain an analog condition to $\sum^{k}_{i=1}\psi_{i}'\in \mathcal{J}$ since 
    \begin{align*}
        \sum^{k}_{i=1}g(ix)=1-\sum^{N}_{i=k+1}g(ix),
    \end{align*} is non-increasing, where $g(ix)$ plays the role of $\psi_{i}'$.
  
Hulse gave a condition in \cite[Theorem 2.2]{hu} that guarantees uniqueness of $g$-measures (equilibrium states) for attractive $g$-functions. In the context of $a$-convex transformations, the analogous condition is: for each $1 \leq k \leq N$,
 \begin{align}\label{HulseCondition}
     \lim_{n\rightarrow \infty}\left[\overline{F}^{n}_{s}\chi_{[0,a_{k}]}(0)-\overline{F}_{s}^{n}\chi_{[0,a_{k}]}(1)\right]=0.
 \end{align} Since $\overline{F}_s^n \chi_{[0, a_i]} \in \mathcal{J}$, equation \eqref{HulseCondition} is stronger than Condition $(B)$, but harder to verify. In contrast, condition $(B)$ focuses on how $|(T^n)'|^{-1}$ decreases to zero near indifferent fixed points. Furthermore, attractive $g$-functions require the constant function $\mathds{1}$ to be the eigenfunction of the transfer operator. 
\end{remark}
As a consequence of Proposition \ref{LasotaYorkeInequality} and the fact that $m_{s}$ is conformal, we obtain Theorem \ref{ACIM}.
\begin{theorem}\label{ACIM}
     There exists a $T$-invariant measure $\mu_{s}$ absolutely continuous with respect to $m_{s}$. Furthermore, the density $g = \frac{d\mu_{s}}{dm_{s}}$ may be chosen from $\mathcal{J}$.
\end{theorem}
\begin{proof}
    Applying Proposition \ref{LasotaYorkeInequality} iteratively to the constant function $\mathds{1}$, we obtain for every $n\in \N$
        \begin{align*}
            V_{I}\overline{F}_{s}^{n}1  \leq \alpha V_{I}\overline{F}^{n-1}_{s}1 + b \leq \cdots \leq \dfrac{b}{1-\alpha}.
        \end{align*} Consider $g_{n}=\frac{1}{n}\sum^{n-1}_{r=0}\overline{F}_{s}^{r}1$. Notice first that $g_{n}\in \mathcal{J}$ since each $\overline{F}_{s}^{r}1\in \mathcal{J}$. Thus, for every $n\in \N$ we have $\|g_{n}\|_{L^{1}_{m_{s}}}=1$ by Proposition \ref{MarkovOp}. We obtain $\|\overline{F}_{s}g_{n}-g_{n}\|_{L^{1}_{m_{s}}}\leq \frac{2}{n}\longrightarrow 0$ as $n \rightarrow \infty$. By Theorem \ref{boundedIneq}
        \begin{align*}
            \|g_{n}\|_{\infty} 
             \leq \frac{1}{n}\sum^{n-1}_{k=0} P^{k}_{s}1(0)+V_{I}P^{k}1 \leq \frac{1}{n}\sum^{n-1}_{k=0} \|P^{k}_{s}1\|_{L^{1}_{m_{s}}}+\dfrac{b}{1-\alpha} = 1+\dfrac{b}{1-\alpha}.
        \end{align*}Then $\{g_{n}\}_{n\in \N}$ converges a bounded variation function $g^{*}$ by Helly's first theorem (see \cite[Theorem 2.3.9]{BG}). Furthermore
        \begin{align*}
            V_{I}g^{*}\leq 1+\dfrac{b}{1-\alpha}.
        \end{align*} Taking a subsequence if necessary, $g^{*}=\lim_{k\rightarrow \infty}g_{n_{k}}$, $g^{*}\geq 0$, $\overline{F}_{s}g^{*}=g^{*}$ and $\|g^{*}\|_{L^{1}_{m_{s}}}=1$. Furthermore $g^{*}\in \mathcal{J}$ for some version in $L^{1}_{m_{s}}$. Defining $\mu_{s}=g^{*}m_{s}$ we claim that $\mu_{s}$ is $T$-invariant. Indeed, for every measurable set $A$
        \begin{align*}
            \mu_{s}(T^{-1}A)  = \int_{T^{-1}(A)}g^{*}dm_{s}=\int_{A}\overline{F}_{s}g^{*}dm_{s}=\int_{A}g^{*}dm_{s}=\mu_{s}(A).
        \end{align*}
    \end{proof}

\begin{lemma}\label{Lemma4.4}
    Let $g$ be the $F_{s}$-invariant function given by Theorem \ref{ACIM}, and $A=\int^{\beta}_{0}gdm_{s}$. Define 
        \begin{align*}
            g_{\beta}(x)=\begin{cases}
                g(x)/A & \text{if $0\leq x\leq \beta,$}\\
                0 & \text{if $\beta<x\leq 1$.}
            \end{cases}
        \end{align*}
    Then $g_{\beta}$ is $\overline{F}_{s}$-invariant.
    \end{lemma}
The proof of Lemma \ref{Lemma4.4} is essentially the same as in \cite[Lemma 4.4]{bmss}, since it only requires the operator properties of $\overline{F}_{s}$ and Proposition \ref{MarkovOp}.

For what is left of this section, we will study the non-empty set of $T$-invariant and absolutely continuous measures with respect to $m_{s}$.
\begin{lemma}\cite[Lemma 3]{hk1}\label{SupportMS}
     Suppose $m_{s}\neq \delta_{\beta}$. Let 
     \begin{align*}
         Y=\bigcap\{F:\text{$F\subseteq [0,1]$ closed with $m_{s}(F)=1$}\}
     \end{align*} be the support of $m_{s}$. Then $T$ can be modified at a finite number of points to satisfy $T(Y) \subseteq Y$ and $T^{-1}\{x\} \subseteq Y$ for all $x \in Y$ except for finitely many points. Furthermore, $Y$ is a compact ordered set, and every non-trivial interval $J \subseteq Y$ has a non-zero $m_{s}$-measure.
\end{lemma}
Hence, up to a change on a finite set (zero measure by Proposition \ref{IKPartition}), $Y$ is a $T$-invariant set. We then focus on $(Y, T)$. Similarly to Definition \ref{BoundedVDefinition}, a complex-valued function $f : Y \to \mathbb{C}$ is said to have \textit{bounded variation} if
   \begin{align*}
       V_{Y}f=\sup\left\{\sum^{k-1}_{i=1}|f(x_{i})-f(x_{i-1})|: \{x_{1},\cdots ,x_{k}\}\subset  Y\right\}<\infty,
   \end{align*} where we abuse notation and use $|\cdot|$ to denote the norm on the complex numbers. Similarly to Section 3, we treat $m_{s}$ as a complex measure satisfying equation \eqref{Msmeasure}, where $F_{s}$ acts on complex-valued continuous functions. Denote by $L^{1}_{m_{s}}(Y, \mathbb{C})$ the set of complex-valued integrable functions over $Y$ with respect to $m_{s}$. We define the \textit{variation} of the equivalence class of $f \in L^{1}_{m_{s}}(Y, \mathbb{C})$ by
   \begin{align*}
      v(f)=\inf\{V_{Y}\widehat{f}: \text{$\widehat{f}$ is a version of $f$ in $L^{1}_{m_{s}}(Y,\mathbb{C})$}\}.
   \end{align*} 
We denote by $BV_{m_{s}}$ the set of functions with bounded variation in $L^{1}_{m_{s}}(Y, \mathbb{C})$. It is known that $BV_{m_{s}}$ is a linear subspace of $L^{1}_{m_{s}}(Y, \mathbb{C})$, but it is not closed with respect to the norm $\|\cdot\|_{L^{1}_{m_{s}}(Y, \mathbb{C})}$. The following defines a norm for $BV_{m_{s}}$ (see \cite[Remark 1]{hk1}):
\begin{align}\label{BV-Norm}
     \|f\|_{v}=\|f\|_{L^{1}_{m_{s}}(Y,\mathbb{C})}+v(f), && \text{for each  $f\in BV_{m_{s}}$.}
\end{align}
\begin{lemma}\label{ConditionalExp}
    Let $f \in L^{1}_{m_{s}}(Y, \mathbb{C})$ be such that $v(f) \leq c$ for some $c > 0$. Let $\{M_{r}\}_{r \in \mathbb{N}}$ be the sequence of positive numbers given by condition $(B)$. Then for every $r \in \mathbb{N}$,
    \begin{align*}
        \int|f-E_{m_{s}}(f|\mathcal{U}_{r})|dm_{s}\leq c\left(M_{r}+m_{s}(I_{(N^{*})_{i=1}^{r}})\right),
    \end{align*} where $E_{m_{s}}(f|\mathcal{U}_{r})$ is the conditional expectation of $f$ with respect to the $\sigma$-algebra $\mathcal{U}_{r}$ generated by $\mathcal{P}^{r}$.
\end{lemma}
\begin{proof}
    By Lemma \ref{SupportMS}, every non-trivial interval in $Y$ has non-zero measure. Thus
    \begin{align*}
        \int |f-E_{m_{s}}(f|\mathcal{U}_{r})|dm_{s} & = \int \left|f-\sum_{I_{(d_{i})_{i=1}^{r}}}\dfrac{ \chi_{I_{(d_{i})_{i=1}^{r}}}}{m_{s}(I_{(d_{i})_{i=1}^{r}})}\int_{I_{(d_{i})_{i=1}^{r}}}fdm_{s}\right|dm_{s}\\
        & \leq \sum_{I_{(d_{i})_{i=1}^{r}}}\int_{I_{(d_{i})_{i=1}^{r}}}\left|f-\dfrac{1}{m_{s}(I_{(d_{i})_{i=1}^{r}})}\int_{I_{(d_{i})_{i=1}^{r}}}fdm_{s}\right|dm_{s},
    \end{align*} Since $\frac{1}{m_{s}(I_{(d_{i})_{i=1}^{r}})} \int_{I_{(d_{i})_{i=1}^{r}}} f \, dm_{s}$ is the expected value of $f$ on the cylinder $I_{(d_{i})_{i=1}^{r}}$,
    \begin{align*}
    \int \left|f-\dfrac{1}{m_{s}(I_{(d_{i})_{i=1}^{r}})}\int_{I_{(d_{i})_{i=1}^{r}}}fdm_{s}\right|dm_{s}& \leq m_{s}(I_{(d_{i})_{i=1}^{r}})V_{I_{(d_{i})_{i=1}^{r}}}(f).
    \end{align*} By condition $(B)$, we obtain
    \begin{align*}
        \int |f-E_{m_{s}}(f|\mathcal{U}_{r})|dm_{s}
        &\leq  M_{r} \sum_{I_{(d_{i})_{i=1}^{k}}\neq I_{(N^{*})_{i=1}^{r}}}V_{I_{(d_{i})_{i=1}^{r}}}(f)+m_{s}(I_{(N^{*})_{i=1}^{r}})V_{I_{(N^{*})_{i=1}^{r}}}\\
        &\leq  M_{r}V_{[0,1]}(f) +m_{s}(I_{(N^{*})_{i=1}^{r}})V_{I_{(N^{*})_{i=1}^{r}}}.
    \end{align*}
    The result follows by taking a suitable version of $f$ and using that $v(f) \leq c$.
\end{proof}

\begin{proposition}\label{PropoBVBanachSpace} If $m_{s}\neq \delta_{\beta}$, then 
    \begin{enumerate}
        \item $(BV_{m_{s}},\|\cdot\|_{v})$ is a Banach space.
        \item $BV_{m_{s}}$ is dense in $L^{1}_{m_{s}}(Y,\mathbb{C})$.
    \end{enumerate}
\end{proposition}
 Lemma \ref{SupportMS}, equation \eqref{BV-Norm}, and Lemma \ref{ConditionalExp} imply Proposition \ref{PropoBVBanachSpace}, using similar arguments as in \cite[Lemma 5]{hk1}. Denote $L^{1}_{m_{s}}(Y)$ as the set of real-valued integrable functions over $Y$.
 \begin{proposition}\label{HypITMT}
     Let $T$ be an $a$-convex map satisfying the hypotheses of Theorem \ref{PrincipalThmChap52}. Then
    \begin{enumerate}
        \item If there is a sequence $(f_{n})\subset BV_{m_{s}}$, $f\in L^{1}_{m_{s}}(Y, \mathbb{C})$ such that 
        \begin{enumerate}
            \item $\lim_{n\rightarrow\infty}\|f_{n}-f\|_{L^{1}_{m_{s}}}=0$,
            \item  for all $n\in\N$ we have $\|f_{n}\|_{BV_{m_{s}}}\leq C$,
        \end{enumerate} then $f\in BV_{m_{s}}$ and $\|f\|_{BV_{m_{s}}}\leq C$.
        \item Iterates of $\overline{F}_{s}$ are bounded, that is, $H=\sup_{n\geq 0}\|\overline{F}_{s}^{n}\|_{BV_{m_{s}}}<\infty$.
        \item (Lasota--Yorke type inequality) There exist $k\geq 1$, $0<r<1$, and $0<R<\infty$ such that for $f\in L^{1}_{m_{s}}(Y, \mathbb{C})$,
        \begin{align*}
            \|\overline{F}_{s}^{k}f\|_{BV_{m_{s}}}\leq r\|f\|_{BV_{m_{s}}}+R\|f\|_{L^{1}_{m_{s}}}.
        \end{align*}
            \item If $B$ is a bounded subset of $(BV_{m_{s}},\|\cdot\|_{v})$, then the closure of $\overline{F}_{s}^{k}B$ is compact in $(L^{1}_{m_{s}}(Y, \mathbb{C}),\|\cdot\|_{L^{1}_{m_{s}}})$.
    \end{enumerate}
 \end{proposition}
 \begin{proof}
    Suppose there is a sequence $\{f_{n}\}\subset BV_{m_{s}}$, a function $f\in L^{1}_{m_{s}}(Y, \mathbb{C})$ such that $\lim_{n\rightarrow\infty}\|f_{n}-f\|_{L^{1}_{m_{s}}}=0$, and, for all $n\in\N$ we have $\|f_{n}\|_{v}\leq C$. Let $\varepsilon>0$ be fixed. By the convergence assumption there exists $N\in \N$ such that for every $n\geq N$
    \begin{align*}
        \|f\|_{L^{1}_{m_{s}}}-\varepsilon\leq \|f_{n}\|_{L^{1}_{m_{s}}} \leq \|f\|_{L^{1}_{m_{s}}}+\varepsilon .
\end{align*}
 Thus,
    \begin{align*}       v(f_{n})= \|f_{n}\|_{v}-\|f_{n}\|_{L^{1}_{m_{s}}} \leq C-(\|f\|_{L^{1}_{m_{s}}}-\varepsilon).
    \end{align*} By Proposition \ref{LasotaYorkeInequality} we obtain
    \begin{align*}
        \|f_{n}\|_{\infty}\leq v(f_{n})+\|f_{n}\|_{L^{1}_{m_{s}}}\leq C-\|f\|_{L^{1}_{m_{s}}}+\varepsilon +\|f_{1}\|_{L^{1}_{m_{s}}}+\varepsilon\leq C+2\varepsilon.
    \end{align*}Thus, we can apply Helly's First Theorem to obtain a sub-sequence $\{ f_{n_{k}}\}$ and a function $f^{*}$ such that $f_{n_{k}}\rightarrow f^{*}$ pointwisely $m_{s}$-a.e.p. But $f_{n_{k}}\rightarrow f$ in $L^{1}_{m_{s}}(Y)$. So, $f=f^{*}$ for $m_{s}$-a.e.p. For any partition $0=t_{0}<t_{1}<\cdots <t_{m-1}<t_{m}=1$, we have
    \begin{align*}
        \sum^{m}_{i=1}|f^{*}(t_{i})-f^{*}(t_{i-1})|&=\lim_{k\rightarrow \infty}\sum^{m}_{i=1}|f_{n_{k}}(t_{i})-f_{n_{k}}(t_{i-1})|\\
        &\leq \lim_{k\rightarrow \infty}V_{I}f_{n_{k}}\\
        &\leq C-\|f\|_{L^{1}_{m_{s}}}+\varepsilon.
    \end{align*}Since this holds for any version of $f^{*}$ in $L^{1}_{m_{s}}(Y, \mathbb{C})$ we get $v(f^{*})\leq C-\|f\|_{L^{1}_{m_{s}}}+\varepsilon$ and
    \begin{align*}
        \|f^{*}\|_{v}\leq C-\|f\|_{L^{1}_{m_{s}}}+\varepsilon +\|f\|_{L^{1}_{m_{s}}}=C+\varepsilon.
    \end{align*} Since $\varepsilon>0$ is arbitrary, $(a)$ is proved. 

    Recall $\|\overline{F}_{s}f\|_{L^{1}_{m_{s}}} =\|f\|_{L^{1}_{m_{s}}}$ and suppose $\|f\|_{v}=\|f\|_{L^{1}_{m_{s}}}+v(f)=1$. Given any partition $\{b_{i}\}^{r}_{i=0}$ of $Y$ we have
    \begin{align*}
        \sum^{r}_{i=1}|\overline{F}_{s}f(b_{i})-\overline{F}_{s}f(b_{i-1})|=\sum^{r}_{i=1}\left|\sum_{j=1}^{N}(\psi_{j}')^{s}(b_{i})f\circ \psi_{j}(b_{i})-(\psi_{j}')^{s}(b_{i-1})f\circ \psi_{j}(b_{i-1})\right|.
    \end{align*}If we add and subtract $(\psi_{j}')^{s}(b_{i})f\circ \psi_{j}(b_{i-1})$ in each summand on the right-hand side,
    \begin{align}\label{Ecu547}
        \sum^{r}_{i=1}|\overline{F}_{s}f(b_{i})-\overline{F}_{s}f(b_{i-1})|&\leq \sum^{r}_{i=1}\sum_{j=1}^{N}(\psi_{j}')^{s}(b_{i})|f\circ \psi_{j}(b_{i})-f\circ \psi_{j}(b_{i-1})|\\
        & + \sum^{r}_{i=1}\sum_{j=1}^{N}f\circ \psi_{j}(b_{i-1})((\psi_{j}')^{s}(b_{i})-(\psi_{j}')^{s}(b_{i-1})).
    \end{align}On the one hand, since $\{\psi_{j}(b_{i}),\psi_{j}(b_{i-1})\}_{j,i}$ defines another partition on $Y$, the first term of the right-hand side of inequality \ref{Ecu547} becomes
    \begin{align}\label{Ecu549}
        \sum^{r}_{i=1}\sum_{j=1}^{N}\|\psi_{j}'\|_{\infty,m_{s}}^{s}|f\circ \psi_{j}(b_{i})-f\circ \psi_{j}(b_{i-1})|
        \leq \|T'\|_{\infty,m_{s}}^{-s}v(f),
    \end{align}where $\|\cdot \|_{\infty,m_{s}}$ denotes the essential supremum norm with respect to $m_{s}$. On the other hand, the second term of the right-hand side of inequality \ref{Ecu547} becomes
    \begin{align*}
        \|f\|_{\infty}\sum^{r}_{i=1}\sum_{j=1}^{N}((\psi_{j}')^{s}(b_{i})-(\psi_{j}')^{s}(b_{i-1}))\leq \|f\|_{\infty}v(\overline{F}_{s}1).
    \end{align*}Since $\|f\|_{v}\leq 1$, inequalities \eqref{Ecu547} and \eqref{Ecu549}, and by Proposition \ref{LasotaYorkeInequality} there exists $0<b<\infty$ such that 
    \begin{align*}
        \sum^{r}_{i=1}|\overline{F}_{s}f(b_{i})-\overline{F}_{s}f(b_{i-1})|\leq \|T'\|_{\infty,m_{s}}^{-s}+b 
    \end{align*}Since $\infty>\sum^{N}_{j=1}(\psi_{j}'(0))^{s}\geq \sup_{i}(\psi_{j}')^{s}$, we obtain $\|T'\|_{\infty,m_{s}}^{-s}<\infty$. Thus, for every $n\in \N$ $$\|\overline{F}_{s}^{n}f\|_{v}\leq \|T'\|_{\infty,m_{s}}^{-s}+b,$$ which proves $(b)$. The Lasota-Yorke inequality is proven as a consequence of Proposition \ref{LasotaYorkeInequality} for $|T'(0)|^{-1}<r <1$. 

    Finally, let $B$ be a bounded subset of $(BV_{m_{s}},\|\cdot \|_{v})$, say $\|f\|_{v}\leq K$ for every $f\in B$ and some $K>0$. Using arguments similar to those in the proof of Proposition \ref{PropoBVBanachSpace}, we show that $\overline{B}$ is compact in $L^{1}_{m_{s}}(Y,\mathbb{C})$ (if $f_{n}\to f$ in $L^{1}_{m_{s}}(Y,\mathbb{C})$, then $\|f\|_{v}\leq K$). From item $(c)$, $\overline{\overline{F}_{s}(B)}=\overline{F}_{s}(\overline{B})$ is also compact in $L^{1}_{m_{s}}(Y)$, since it is bounded and Helly's First Theorem applies.
\end{proof}
  Let $E $ be the set of all $\overline{F}_{s}$-invariant non-negative functions in $L^{1}_{m_{s}}(Y)$ which is non-empty by Theorem \ref{ACIM}. Let $E_{1} = E \cap \left\{ f : \|f\|_{L^{1}_{m_{s}}} = 1 \right\}$ be the set of $\overline{F}_{s}$-invariant densities of $m_{s}$. Proposition \ref{HypITMT} says the operator $\overline{F}_{s}:BV_{m_{s}} \rightarrow BV_{m_{s}}$ satisfies all the hypotheses of the Ionescu-Tulcea and Marinescu Theorem \cite{tm} (see \cite[Section 7.1]{BG}) for the Banach spaces $BV_{m_{s}}$ and $L^{1}_{m_{s}}(Y, \mathbb{C})$. Thus, $E \subset BV_{m_{s}}$ is a convex set and is finite-dimensional. Furthermore, $E_{1}$ has a finite number of extreme points with disjoint supports spanning $E$ (see \cite[Proposition 7.2.3]{bg}). We obtain Theorem \ref{SpectralDescomp} following \cite[Chapter 7]{bg}.
\begin{theorem}\label{SpectralDescomp}
    Let $T$ be an $a$-convex map satisfying the hypotheses of Theorem \ref{PrincipalThmChap52}. Then $T$ has finitely many $T$-invariant ergodic measures $\{\mu_{1},...,\mu_{n}\}$ absolutely continuous with respect to $m_{s}$. Furthermore, each $\frac{d\mu_{i}}{dm_{s}}\in E_{1}$, and the set $\{\frac{d\mu_{1}}{dm_{s}},\cdots,\frac{d\mu_{n}}{dm_{s}}\}$ spans the space of $\overline{F}_{s}$-invariant densities of bounded variation.
\end{theorem}
Indeed, it is possible to have more than one ACIM for $m_{s}$ (see \cite[Example 6.1]{bmss} when $m_{1}$ is the Lebesgue measure). 
\section{Equilibrium states}
In this section we prove that the $T$-invariant measure $\mu_{s}$ from Theorem \ref{ACIM} is an equilibrium state for the potential $-s\log |T'|$ on $Y$, i.e.,
\begin{align}\label{PressureVariationalPrinciple}
    h(\mu_{s})-\int s\log |T'|d\mu_{s}=\sup\left\{h(\nu)-\int s\log |T'|d\nu:\text{$\nu$ is $T$-invariant on $Y$}\right\},
\end{align}
where $Y\subseteq[0,1]$ is the support of the conformal measure $m_{s}$ from Section 4. The value on the right-hand side of equation \eqref{PressureVariationalPrinciple} is the topological pressure of $T$ for the potential $-s\log |T'|$ over $Y$, denoted by $P(s)$.
\begin{proposition}\cite[Theorem 8.1.2]{BG}\label{Thm8.1.2}
     Suppose $m_{s} \neq \delta_{\beta}$. If $f$ is an $\overline{F}_{s}$-invariant density of bounded variation in $L^{1}_{m_{s}}(Y)$, then the support of $f$ is open for $m_{s}$-almost every point.
\end{proposition}
Let $g$ be the $\overline{F}_{s}$-invariant density from Theorem \ref{ACIM}. Recall that $g_{\beta} = (g / A) \chi_{[0, \beta]}$ is $\overline{F}_{s}$-invariant, where $A = \int_0^{\beta} g \, dm_{s}$ by Lemma \ref{Lemma4.4}.
\begin{proposition}\label{PositiveDensity}
    Suppose $m_{s} \neq \delta_{\beta}$ and $\psi_{j}'(\beta)>0$ for some $1\leq j < N$. Then $g_{\beta}$ is the only $\overline{F}_{s}$-invariant density whose support is $[0, \beta]$ and $g_{\beta}|_{[0,\beta]} \geq c$ for some $c > 0$.
\end{proposition}
\begin{proof}
    Recall that $g_{\beta}\in BV_{m_{s}}$  since $g_{\beta} \in \mathcal{J}$. By Theorem \ref{SpectralDescomp}, we have, for some constants $\alpha_{1},...,\alpha_{n}$, that
     \begin{align*}
         g_{\beta}=\sum^{n}_{i=1}\alpha_{i}\dfrac{d\mu_{i}}{dm_{s}}.
     \end{align*} The support of each $\frac{d\mu_{i}}{dm_{s}}$ is an open set for $m_{s}$-almost every point by Proposition \ref{Thm8.1.2}. There exists $x > 0$ and $1 \leq q \leq k$ such that $(0, x)$ supports $\frac{d\mu_{q}}{dm_{s}}$. Then expansiveness on $(0, a_{1})$ implies that $(0, \beta) \subseteq T^{r}(0, x)$ for some $r \in \mathbb{N}$ (see Remark \ref{SomeIterationExpansive}). Therefore, $(0, \beta)$ must be in the support $\frac{d\mu_{q}}{dm_{s}}$. Since the $\overline{F}_{s}$-invariant density functions from Theorem \ref{SpectralDescomp} have disjoint supports (as the respective measures are ergodic), then $g_{\beta} = \frac{d\mu_{q}}{dm_{s}}$, and $(T, g_{\beta} m_{s})$ is ergodic. 
     
     Recall $g_{\beta} $ attains its minimum at $g_{\beta}(\beta)$. On the other hand, by $\overline{F}_{s}$-invariance, we have
     \begin{align*}
         g_{\beta}(\beta)=\dfrac{1}{A\gamma}\sum^{N}_{i=1}(g\circ \psi_{i}(\beta))(\psi_{i}'(\beta))^{s}.
     \end{align*} Now suppose that $\lim_{x \to \beta^{-}} g_{\beta}(x) = 0$. Then,
     \begin{align*}
         \lim_{x\rightarrow \beta^{-}}g_{\beta}(x)\geq \lim_{x\rightarrow \beta^{-}}\dfrac{(g\circ \psi_{j}(x))(\psi_{j}'(x))^{s}}{A\gamma}=0.
     \end{align*}
    Thus, by our hypothesis $g_{\beta}(a_{j}) = 0$ since $\lim_{x \to \beta^{-}} g \circ \psi_{j}(x) = g(a_{j})$. This implies that the support of $g_{\beta}$ is in $[0, a_{j}]$. Let $\mu_{s}=g_{\beta}m_{s}$. Therefore, $\mu_{s}((a_{j}, \beta)) = 0$.
    
    On the other hand, for any $x \in (0, a_{1}]$, there exists $N \in \mathbb{N}$ and $1\leq i \leq j$ such that for each $k\geq N$ we have $T_{i}^{-1}(a_{j}) \leq T^{k}(x)$ by Remark \ref{SomeIterationExpansive}. Thus, $0 = \mu_{s}(T^{-k+1} [a_{j}, \beta)) \geq \mu_{s}((x, \beta))$ by $T$-invariance. Since $x$ is arbitrary, we obtain $\mu_{s}((0, \beta)) = 0$. Therefore, $g_{\beta} m_{s} = \delta_{0}$, where $\delta_{0}$ is the Dirac delta measure at $x = 0$. This leads to a contradiction, since $m_{s}$ assigns zero measure on singleton sets by Proposition \ref{IKPartition}.
\end{proof}
By Proposition \ref{PositiveDensity}, the function $\overline{g}=|T'|^{-s}\dfrac{g_{\beta}}{g_{\beta}\circ T}$ is well defined whenever $\psi_{j}'(\beta)>0$ for some $1\leq j < N$. Define $G_{s}:C(Y)\rightarrow C(Y)$ by
\begin{align}\label{Ecu5413}
    G_{s}(f)=\dfrac{1}{\gamma}\sum^{N}_{i=1}(\overline{g}f)\circ \psi_{i}=\dfrac{1}{g_{\beta}\gamma}\sum^{N}_{i=1}(\psi_{i}')^{s}(g_{\beta}f\circ \psi_{i}),
\end{align} where $\gamma= m_{s}(F_{s}\mathds{1})$ (see equation \eqref{Msmeasure}).

\begin{lemma}\label{Lemma 15}
   Suppose $\psi_{j}'(\beta)>0$ for some $1\leq j < N$. If $\mu$ is a $T$-invariant probability measure on $Y$, then 
    \begin{align*}
        \int \log(\overline{g})d\mu-\log(\gamma)  =-s\int \log|T'|d\mu.
    \end{align*}
\end{lemma}
The proof of Lemma \ref{Lemma 15} follows \cite[Lemma 15]{hk1}, as it relies on measure-theoretic arguments. Consequently, for any $T$-invariant measure $\mu$ on $Y$, we have
    \begin{align}\label{ASTERISCO}
        h_{\mu}+\int \log \overline{g}d\mu =h_{\mu}-s\int \log|T'|d\mu-\log(\gamma).
    \end{align}
Therefore, $T$-invariant measures that attain the supremum in equation \eqref{ASTERISCO} are equilibrium states for both $\overline{g}$ and $-s\log |T'|$. The following theorem is proved in the context of $g$-measures, but it relies on Lemma \ref{Lemma 15} and the Rokhlin formula, since $T$ is locally invertible on $Y$ (see \cite[Theorem 2.1]{wp}).
\begin{theorem}\label{Ledrapierthm}
    Suppose $\psi_{j}'(\beta)>0$ for some $1\leq j < N$. Let $\mu$ be a Borel probability measure on $Y$. The following statements are equivalent:
    \begin{enumerate}
        \item The dual operator $G^{*}_{s}:C^{*}(Y)\rightarrow C^{*}(Y)$ has $\mu$ as a fixed point.
        \item  $\mu$ is a $T$-invariant measure on $Y$ and is an equilibrium state for $\log(\overline{g})$.
    \end{enumerate}
\end{theorem}

\textit{Proof of Theorem \ref{PrincipalThmChap52}.}
First, assume $\psi_{i}'(\beta)=0$ for all $1\leq i < N$. This implies that some iteration is expansive, so by \cite[Theorem 6]{hk1} $\mu_{s}$ is an equilibrium state for $-s\log|T'|$. Suppose now $\psi_{j}'(\beta)>0$ for some $1\leq j < N$. Recall that by Proposition \ref{IKPartition} $\psi_{N^{*}}(\beta^{-})<1$ or $\psi_{N^{*}+1}'(\beta^{+})=1$ implies $m_{s}\in \mathcal{M}_{1}([0,1])$. Since $\mu_{s} = g_{\beta} m_{s}$ and $\overline{F}^{*}_{s} m_{s} = m_{s}$, we obtain for every $f \in C(Y)$
    \begin{align*}
        \mu_{s}(G_{s}f)=\int \frac{1}{\gamma}\sum^{N}_{i=1}(\psi_{i}')^{s}g_{\beta}f\circ \psi_{i}dm_{s}=\int \overline{F}_{s}(g_{\beta}f)dm_{s}=\int fg_{\beta}dm_{s}=\mu_{s}(f).
    \end{align*}Therefore, $\mu_{s}$ is an equilibrium state of $-s \log |T'|$, according to Theorem \ref{Ledrapierthm}. By Proposition \ref{PositiveDensity}, the uniqueness follows when $\beta = 1$ and $m_{s}\neq \delta_{\beta}$. \qedsymbol
\begin{example}
    Recall that we showed that transformations in Examples \ref{ExampleParabolicAconvex} and \ref{ExampleParabolicNonMarkovianAconvex} satisfy condition $(B)$ for each $0 < s < 1$ such that $m_{s} \neq \delta_{1}$, and condition $(C)$ for all $0 < s \leq 1$ (Example \ref{ConditionBExamplesInfFixPt}). We now show that the function $s \mapsto P(s)$ is non-increasing in $(0, \infty)$, constant and equal to zero for $s \geq 1$, and differentiable except for at most countably many points. 
    
    Using equation \eqref{PressureVariationalPrinciple}, we deduce that $s \mapsto P(s)$ is non-increasing. Moreover, since $|T'(x)| = 1$ if and only if $x = 1$, we have, for every $\varepsilon > 0$, that $P(s + \varepsilon) = P(s)$ if and only if for each $t\in (s,s+\varepsilon)$ an equilibrium state of $-t\log |T'|$ is $\delta_{1}$.
    
    However, for $s=1$, there is an equilibrium state different from $\delta_{1}$, namely the ACIM with respect to Lebesgue measure (\cite[Theorem 1.1]{bmss}). The proof of Theorem \ref{Ledrapierthm} in \cite[Theorem 2.1]{wp} shows that an equilibrium state $\mu$ for $\log(\overline{g})$ satisfies
    \begin{align*}
        h_{\mu} + \int \log(\overline{g}) \, d\mu = 0.
    \end{align*} Since $\gamma= m_{1}(F_{1} \mathds{1}) = 1$, we conclude that $P(1) = 0$ by Lemma \ref{Lemma 15}. Finally, the pressure function is differentiable except for, at most, countably many points since it is a convex function (\cite[Theorem 3.6.2]{pu}).

    We can also deduce the existence of equilibrium states in Example \ref{ExampleParabolicAconvex} by inducing and obtaining the exact description of the pressure function. However, we will not develop this point here (see \cite{sa, mf}). Nevertheless, our method allows us to study the thermodynamic formalism of $a$-convex transformations with non-Markov partitions, such as in Example \ref{ExampleParabolicNonMarkovianAconvex}, which, to our knowledge, cannot be achieved by inducing.

\end{example}

\section*{Acknowledgements}
I thank my Ph.D. advisor, Godofredo Iommi, for suggesting this problem, his patience, and many helpful remarks. This research was supported by ANID Doctorado Nacional 21210037.

\end{document}